\documentclass[leqno]{amsart}
\usepackage{amsmath}
\usepackage{amssymb}
\usepackage{amsthm}
\usepackage{enumerate}
\usepackage[mathscr]{eucal}
\usepackage{graphicx}
\theoremstyle{plain}
\newtheorem{theorem}{Theorem}[section]
\newtheorem{prop}[theorem]{Proposition}
\newtheorem{cor}{Corollary}[theorem]
\newtheorem{lemma}{Lemma}[section]

\theoremstyle{definition}

\newtheorem{remark}{Remark}[section]

\newtheorem*{acknowledgement}{\textnormal{\textbf{Acknowledgements}}}
\usepackage[pagewise]{lineno}

\begin{document}
\title[Orthogonality of operators on complex Banach spaces]{Orthogonality of bounded linear operators on complex Banach spaces}
\author[ Kallol Paul, Debmalya Sain,  Arpita Mal and Kalidas Mandal ]{ Kallol Paul, Debmalya Sain, Arpita Mal and Kalidas Mandal}

\newcommand{\acr}{\newline\indent}

\address[Paul]{Department of Mathematics\\ Jadavpur University\\ Kolkata 700032\\ West Bengal\\ INDIA}
\email{kalloldada@gmail.com}

\address[Sain]{Department of Mathematics\\ Indian Institute of Science\\ Bengaluru 560012\\ Karnataka \\India\\ }
\email{saindebmalya@gmail.com}

\address[Mal]{Department of Mathematics\\ Jadavpur University\\ Kolkata 700032\\ West Bengal\\ INDIA}
\email{arpitamalju@gmail.com}

\address[Mandal]{Department of Mathematics\\ Jadavpur University\\ Kolkata 700032\\ West Bengal\\ INDIA}
\email{kalidas.mandal14@gmail.com}

\thanks{ The research of the second author is sponsored by Dr. D. S. Kothari Postdoctoral fellowship. The third and fourth author would like to thank UGC and CSIR, Govt. of India respectively for the financial support.} 


\begin{abstract}
We study Birkhoff-James orthogonality of bounded linear operators on complex Banach spaces and obtain a complete characterization of the same. By means of introducing new definitions, we illustrate that it is possible in the complex case, to develop a study of orthogonality of bounded (compact) linear operators, analogous to the real case. Furthermore, earlier operator theoretic characterizations of Birkhoff-James orthogonality in the real case, can be obtained as simple corollaries to our present study. In fact, we obtain more than one equivalent characterizations of Birkhoff-James orthogonality of compact linear operators in the complex case, in order to distinguish the complex case from the real case. We also study the left symmetric linear operators on complex two-dimensional $l_p$ spaces. We prove that $ T $  is a left symmetric linear operator on $ \ell_p^2{(\mathbb{C})}$ if and only if $ T $ is the zero operator. 
\end{abstract}

\maketitle
\section{Introduction.} 

The notion of Birkhoff-James orthogonality (B-J orthogonality) plays a very important role in the geometry of Banach spaces. In \cite{J},  James illustrated the role of  B-J orthogonality in characterizing geometric properties like smoothness, strict convexity etc. of the space. It is quite straightforward to observe that the notion of B-J orthogonality extends to the space of all bounded linear operators on a Banach space.  The role of B-J orthogonality in the study of geometry of Banach spaces has been explored by several researchers, from various points of view. We refer the readers to \cite{C,CSW,K,T}, and the references therein, for a detailed study in this regard. Recently, in \cite{S}, Sain characterized B-J orthogonality of linear operators on finite-dimensional real Banach spaces. Although B-J orthogonality can be defined for either real or complex Banach spaces, till now most of the operator theoretic study of B-J orthogonality \cite{S,SP} has been conducted exclusively in the context of real Banach spaces. In this paper, our aim is to initiate an analogous study of B-J orthogonality of linear operators in the complex case and to obtain a characterization of the same. It is interesting to observe that the results already known in the context of real Banach spaces follow quite easily from these new results. It is in this sense, that our present study can be considered as an extension of our earlier studies \cite{S,SGP}. Without further ado, let us establish the relevant notations and terminologies to be used throughout the paper.

\medskip

Let $\mathbb{X},~ \mathbb{Y}$ be complex Banach spaces. Let $B_{\mathbb{X}} = \{x \in \mathbb{X} \colon \|x\| \leq 1\}$ and
$S_{\mathbb{X}} = \{x \in \mathbb{X} \colon \|x\|=1\}$ be the unit ball and the unit sphere of $\mathbb{X}$, respectively. Let $\mathbb{L}(\mathbb{X},\mathbb{Y})(\mathbb{K}(\mathbb{X},\mathbb{Y}))$ denote the Banach space of all bounded (compact) linear operators from $\mathbb{X}$ to $\mathbb{Y}$, endowed with the usual operator norm. We write $\mathbb{L}(\mathbb{\mathbb{X}, \mathbb{Y}})= \mathbb{L}(\mathbb{X})$ and $\mathbb{K}(\mathbb{X},\mathbb{Y})=\mathbb{K}(\mathbb{X})$ if $\mathbb{X}= \mathbb{Y}$.

For any two elements  $x,y \in {\mathbb{X}}$, $x$ is said to be B-J orthogonal to $y$, written as $x \perp_B y$, if $ \|x+\lambda y\|\geq\|x\|$
$\forall \lambda \in \mathbb{C}$.

Similarly, for any two elements  $T,A \in \mathbb{L}(\mathbb{X})$,
$T$ is said to be B-J orthogonal to $A$, written as $T \perp_B A$,
if $ \|T + \lambda A\| \geq \|T\|$ $\forall \lambda \in  \mathbb{C}$. 

For a linear operator $T$ defined on a Banach space $\mathbb{X}$, let $M_T$ denote the collection of all unit vectors in $\mathbb{X}$ at which $T$ attains norm, i.e,
\[ M_T= \{  x \in S_ \mathbb{X} \colon\|Tx\| = \|T\| \}. \]

In order to characterize B-J orthogonality of bounded linear operators on finite-dimensional real Banach spaces, Sain \cite{S} introduced the notions of $x^+ $ and $x^-$ in the following way: \\
For any two elements $ x, y $ in a real Banach space $ \mathbb{X}, $ let us say that $ y \in x^{+} $ if $ \| x + \lambda y \| \geq \| x \| $ for all $ \lambda \geq 0. $ Following similar motivations, we say that $ y \in x^{-} $ if $ \| x + \lambda y \| \geq \| x \| $ for all $ \lambda \leq 0. $ Using these notions, Sain \cite{S} characterized B-J orthogonality of linear operators defined on finite-dimensional real Banach spaces. 

\begin{theorem}[Theorem 2.2 of \cite{S}]
	Let $ \mathbb{X} $ be a finite-dimensional real Banach space. Let $T, A \in \mathbb{L}(\mathbb{X}).$ Then  $T \perp_B A $  if and only if there exists $x, y \in M_T $ such that $ Ax \in Tx^{+} $ and $ Ay \in Ty^{-} $.
\end{theorem}

The characterization of B-J orthogonality of bounded linear operators on finite-dimensional real Banach spaces was later extended to real Banach spaces of arbitrary dimension in \cite{SPM}. The following two notions were introduced in \cite{SPM}, in order to accomplish the goal.\\
For any two elements $x,~y$ in a real Banach space $\mathbb{X}$ and $\epsilon \in [0,1)$, we say that $y \in x^{+ (\epsilon)}$ if $\|x+ \lambda y\| \geq \sqrt{1-{\epsilon}^2}\|x\|$ for all $\lambda \geq 0$. Similarly, we say that  $y \in x^{- (\epsilon)}$ if $\|x+ \lambda y\| \geq \sqrt{1-{\epsilon}^2}\|x\|$ for all $\lambda \leq 0$.\\
 In this paper, in order to obtain an analogous result for complex Banach spaces, let us introduce the following notions:  \\ 

Let  $ x \in \mathbb{X}$ and $ U = \{ \alpha \in \mathbb{C} : | \alpha | = 1, ~\arg \alpha \in [0,\pi) \}.$ For $ \alpha \in U $  and $\epsilon \in [0,1)$ define 

\[ x_\alpha^{+}=\{y \in\mathbb{X}: \|x+\lambda y\|\geq\|x\|~ \forall~ \lambda = t\alpha,  t\geq 0 \}. \]
\[ x_\alpha^{-}=\{y \in\mathbb{X}:\|x+\lambda y\|\geq\|x\|~ \forall~ \lambda = t\alpha,  t\leq 0 \}.\] 
\[ x_\alpha^{\perp}  =  \{y \in\mathbb{X}:\|x+\lambda y\|\geq\|x\|~ \forall~ \lambda = t\alpha,   t\in\mathbb{R}\}.\]
\[ x_\alpha^{+(\epsilon)}=\{y \in\mathbb{X}: \|x+\lambda y\|\geq \sqrt{1-{\epsilon}^2}\|x\|~ \forall~ \lambda = t\alpha,  t\geq 0 \} .\]
\[ x_\alpha^{-(\epsilon)}=\{y \in\mathbb{X}:\|x+\lambda y\|\geq \sqrt{1-{\epsilon}^2}\|x\|~ \forall~ \lambda = t\alpha,  t\leq 0 \}.\] 
\[ x_\alpha^{\perp (\epsilon)}  =  \{y \in\mathbb{X}:\|x+\lambda y\|\geq \sqrt{1-{\epsilon}^2}\|x\|~ \forall~ \lambda = t\alpha,   t\in\mathbb{R}\}.\]

If $ \beta = e^{i\pi} \alpha $ then we define $ x_\beta^{+} = x_\alpha^{-}, ~ x_\beta^{-} = x_\alpha^{+} $, $x_\beta^{\perp} = x_\alpha^{\perp}$, $ x_\beta^{+(\epsilon)} = x_\alpha^{-(\epsilon)},~ x_\beta^{-(\epsilon)} = x_\alpha^{+(\epsilon)} $ and $x_\beta^{\perp(\epsilon)} = x_\alpha^{\perp(\epsilon)}$. \\
If $ y \in x_\alpha^{\perp} $ then we write $ x \bot_{\alpha} y.$  Let us define the notions of $x^+$, $x^-$ and $x^{\perp}$ in a complex Banach space in the following way:
 \[x^{+}   =  \bigcap \{ x_\alpha^{+} : \alpha \in U \} .\]
 \[x^{-}  =  \bigcap \{ x_\alpha^{-} : \alpha\in U \} .\]
 \[x^{\perp}  =  \bigcap \{ x_\alpha^{\bot} : \alpha\in U \} .\]
 If the space $\mathbb{X}$ is a real Banach space, then we must have, $\alpha \in U$ implies that $\alpha=1$. Therefore, $ x_\alpha^{+} = x^+, $  $ x_\alpha^{-} = x^- $  and $ x_\alpha^{\perp} = x^{\bot}.$
 
\smallskip
In this paper, we completely characterize B-J orthogonality of bounded linear operators between complex Banach spaces. The characterization assumes a particularly nice form if the domain space is reflexive and the operators are compact. In order to illustrate the importance of our study, we show that earlier characterizations of operator B-J orthogonality \cite{S,SPM} in the real case follow as simple corollaries to our present study. Next we consider the left symmetry of B-J orthogonality of linear operators defined on a finite-dimensional complex Banach space $ \mathbb{X}. $ For an element $ x \in  \mathbb{X}, ~ x $ is said to be left symmetric (with respect to B-J orthogonality) if $ x \perp_{B} y $ implies $ y \perp_{B} x $ for any $ y \in \mathbb{X}.$ A study of left symmetric linear operators on a finite-dimensional real Banach space was carried out in \cite{SGP}. In this paper, we study the left symmetric linear operators on complex two-dimensional $l_p$ spaces. We prove that $ T $  is a left symmetric linear operator on $ \ell_p^2{(\mathbb{C})}$ if and only if $ T $ is the zero operator. 

\section{Main Results}   
Let us begin with two easy propositions, that would be useful in obtaining the desired characterization of B-J orthogonality of bounded linear operators between complex Banach spaces.

\begin{prop}\label{prop:prop1 } 
Let $\mathbb{X}$ be a complex Banach space, $x, y \in \mathbb{X} $ and $ \alpha \in U.$  Then the following are true:
	
	(i) Either $y \in x_{\alpha}^{+}$ or $y \in x_{\alpha}^{-}$.
	
	(ii)  $x\perp_{\alpha} y$ if and only if  $y \in x_\alpha^{+}$ and $y \in x_\alpha^{-}$. 
	
	(iii)  $y \in x_\alpha^{+}$ implies that $\eta y \in (\mu x)_\alpha^{+}$ for all $\eta, \mu>0$. 	
	
	(iv)  $y \in x_\alpha^{+}$ implies that $ - y \in  x_\alpha^{-}$ and $y\in(-x)_\alpha^{-}$. 
	
	(v)  $y \in x_\alpha^{-}$ implies that $\eta y \in (\mu x)_\alpha^{-}$ for all $\eta, \mu>0$. 
	
	(vi)   $y \in x_\alpha^{-}$ implies that $ - y \in  x_\alpha^{+}$ and $y\in(-x)_\alpha^{+}$. 
	
	(vii) $y \in x_\alpha^{+}$ implies that $\beta y \in (\beta x)_\alpha^{+}$ for all $\beta \in \mathbb{C}$.
	
	(viii) $y \in x_\alpha^{-}$ implies that $\beta y \in (\beta x)_\alpha^{-}$ for all $\beta \in \mathbb{C}$.
\end{prop}	
\begin{proof}
$(i)$ If $ y \notin x_\alpha^+ $ then we show that $ y \in x_{\alpha}^-.$ Since $ y \notin x_\alpha^+ $ we have $\|x+\lambda_{0} y\|<\|x\|$ for some $\lambda_{0}=t_{0}\alpha$ with $t_{0}>0$. Let  $\lambda =t \alpha$ with $t< 0.$ Then there exists $ s \in [0,1] $ such that 
	$x= s(x+\lambda_{0}y) + (1-s)(x+\lambda y) \Rightarrow \|x\|\leq s\|x+\lambda_{0}y\| +(1-s)\|x+\lambda y\| \Rightarrow  \|x\|< s\|x\| +(1-s)\|x+\lambda y\| \Rightarrow \|x\|< \|x+\lambda y\|$. 	 Therefore,  $ \|x\|\leq \|x+\lambda y\|  ~\forall \lambda=t\alpha$ with $t\leq 0 \Rightarrow y\in x_{\alpha}^-$. \\
	
The proofs of $(ii)-(viii)$ can be easily completed using similar approach.
\end{proof}

\begin{prop}\label{prop:prop2 } 
Let $\mathbb{X}$ be a complex Banach space and $x, y \in \mathbb{X}$. Then the following are true:

	(i)  	$x\perp_{B} y$ if and only if  $y \in x^{+}$ and $y \in x^{-}$.
	
	(ii)  $y \in x^{+}$ implies that $\eta y \in (\mu x)^{+}$ for all $\eta, \mu>0$.
	
	(iii)  $y \in x^{+}$ implies that $ - y \in  x^{-}$ and $y\in(-x)^{-}$.
	
	(iv)  $y \in x^{-}$ implies that $\eta y \in (\mu x)^{-}$ for all $\eta, \mu>0$.
	
	(v)   $y \in x^{-}$ implies that $ - y \in  x^{+}$ and $y\in(-x)^{+}$.
\end{prop}	
\begin{proof}

	 	$ (i)$ The proof follows from the definitions of $x^+ $ and $x^-$.

	 $(ii)$ Let $y \in x^+$. Then $y \in x_{\alpha}^+$ for each $\alpha $ with $ \arg \alpha \in [0, \pi).$ We show that $ \eta y \in (\mu x)_{\alpha}^+ $ for each $\alpha $ with $ \arg \alpha \in [0, \pi).$  Now, $\|\mu x + (t \alpha) \eta y\|= |\mu|\|x+    (\frac{t \eta}{\mu}) \alpha y\| \geq|\mu| \|x\|= \|\mu x\|~ \forall ~t,\mu, \eta> 0$  and so $ \eta y\in (\mu x)_\alpha^{+} ~\forall ~ \mu , \eta> 0$. 	Thus, $ \eta y \in (\mu x)^+   ~\forall ~ \mu , \eta> 0$.
	  
	  $(iii)$ Suppose $y \in x^{+}$. Then for each $\alpha $ with $ \arg \alpha \in [0,\pi), ~\|x+ t \alpha y\|\geq\|x\| ~\forall~ t\geq 0.$ So $ \|x+ (-t) \alpha (-y)\| \geq \|x\| ~\forall~ t \geq 0. $ This shows that $-y \in x_{\alpha}^{-}$ for each $\alpha $ with $ \arg \alpha \in [0,\pi)$ and so $ -y \in x^{-}$.
	  
	  Again, for each $\alpha $ with $ \arg \alpha \in [0,\pi), ~\|x+ t \alpha y\|=\|(-x) + (-t) \alpha (y)\|\geq\|x\|=\|-x\| ~\forall~ t\geq 0$.   This shows that $y \in (-x)_{\alpha}^{-}$ for each $\alpha $ with $ \arg \alpha \in [0,\pi)$ and therefore, $ y \in (-x)^{-}$.
	  
	  $(iv)$ Follows similarly as $(ii).$
	  
	 $ (v)$ Follows similarly as $(iii).$
\end{proof}	

Let us now obtain the promised characterization theorem.

\begin{theorem}\label{theorem:bounded}
	Let $\mathbb{X}$ and $\mathbb{Y}$ be two complex Banach spaces. Let $T \in \mathbb{L} (\mathbb{X}, \mathbb{Y})$ be non-zero. Then for  any $A \in \mathbb{L} (\mathbb{X}, \mathbb{Y}) ,~ T \bot_B A$ if and only if either $(a)$ or  $(b)$ holds: \\
	$(a)$ There exists a sequence $\{x_n\}$ in $S_\mathbb{X}$ such that $\| Tx_n\| \rightarrow \|T\|$ and $\|Ax_n\| \rightarrow 0$ as $n\rightarrow \infty$. \\
	$(b)$ For each $\alpha \in U$, there exist two sequences $\{x_n= x_n(\alpha)\},~\{y_n=y_n(\alpha)\}$ in $S_\mathbb{X}$ and two sequences of positive real numbers $\{\epsilon_n=\epsilon_n(\alpha)\}$, $\{\delta_n=\delta_n(\alpha)\}$ such that \\ 
	(i) $ \epsilon_n \rightarrow 0$, $\delta_n\rightarrow 0$ as $n\rightarrow \infty$,\\
	(ii) $\| Tx_n\| \rightarrow \|T\|$ and $\|Ty_n\|\rightarrow\|T\|$ as $n\rightarrow \infty$,\\
	(iii) $Ax_n\in(Tx_n)^{+(\epsilon_n)}_{\alpha}$ and $Ay_n\in(Ty_n)^{-(\delta_n)}_{\alpha}$ for all  $n\in \mathbb{N}$.
\end{theorem}
\begin{proof}
	We first prove the easier sufficient part.\\
	Suppose $(a)$ holds. Now, for any  $\lambda \in \mathbb{C}$, $\|T + \lambda A\| \geq \|Tx_n + \lambda Ax_n\|\geq \|Tx_n\|- |\lambda | \|Ax_n\| \rightarrow \|T\|$ as $n \rightarrow \infty$. Therefore, $T \bot_{B} A$. \\
	Now,  suppose $(b)$ holds. Let $\lambda \in \mathbb{C}$. Then $\lambda= t \alpha$ for some $\alpha \in U$ and $t \in \mathbb{R}$. If $t \geq 0$ then  $ Ax_n\in(Tx_n)^{+(\epsilon_n)}_{\alpha} $ for all $n\in \mathbb{N},$ gives that  $\|Tx_n+\lambda Ax_n\|=\|Tx_n+ t \alpha Ax_n\| \geq \sqrt[]{1-\epsilon_n^2}~\|Tx_n\|$.\\
	This implies, $ \|T+\lambda A\|\geq\|(T+\lambda A)x_n\|=\|Tx_n+ t \alpha Ax_n\|  \geq\sqrt[]{1-\epsilon_n^2} \|Tx_n\|$. Since $ \epsilon_n \rightarrow 0 $ and $ \|Tx_n\| \rightarrow \| T \| $ as $ n \rightarrow \infty, $ we obtain, 
	\[ \|T+\lambda A \|\geq \|T\|. \]
	Similarly, if $t \leq 0$ then $ Ay_n\in(Ty_n)^{-(\delta_n)}_{\alpha} $ for all $n\in \mathbb{N}$ implies that
	\[ \|T+\lambda A \|\geq \|T\|. \]
	This completes the proof of the sufficient part. \\
	Let us now prove the comparatively trickier necessary part. \\
	Suppose $(a)$ does not hold. \\
	Without loss of generality let us assume that $\|A\| \leq 1$. \\
	Since $T \bot_{B} A$, for any nonzero scalar $\lambda$, $\|T + \lambda A\| \geq \|T\|$. In particular, for $\alpha \in U$ and for each $n \in \mathbb{N}$, 
	\[\|T + \frac{\alpha}{n} A\| > \|T\| - \frac{1}{n^3}. \]
	Therefore, for each $n \in \mathbb{N}$, there exists a sequence $\{x_n\} $ in $ S_{\mathbb{X}}$ such that $\|(T+ \frac{\alpha}{n} A) x_n \| > \|T\|- \frac{1}{n^3} \geq \|Tx_n\|- \frac{1}{n^3}$. \\
	We claim that $\|Tx_n\|\rightarrow \|T\|$. Indeed, $\|Tx_n\|=\|(T+\frac{\alpha}{n}A)x_n - \frac{\alpha}{n}Ax_n\|$ $\geq \|(T+\frac{\alpha}{n}A)x_n\| - |\frac{\alpha}{n}| \|Ax_n\|$ $>\|T\|- \frac{1}{n^3} - \frac{1}{n} \|A\|$ $\rightarrow \|T\|$ as $n \rightarrow \infty. $ Since $ x_n \in S_{\mathbb{X}},$ $ \| Tx_n \| \leq \| T \|  $. This proves our claim. \\
	Since $(a)$ does not hold, we assume that, $\inf \limits_{n \in \mathbb{N}}\|Ax_n\|= c > 0$. \\
	Choose $n_1 \in \mathbb{N}$ such that $n_1 > \frac{2 \|T\|}{c}$. Since $\|Tx_n\| \rightarrow \|T\|>0,$ there exists $n_2\in \mathbb{N}$ such that $\|Tx_n\| > \frac{\|T\|}{2} > 0$ for all $n \geq n_2$. Choose $n_3 \in \mathbb{N}$ such that $n_3 > \frac{2}{\|T\|}$. Let $n_0 =$ max$\{ n_1,n_2, n_3\}$. Then for all $n\geq n_0$, $0< \frac{1}{n \|Tx_n\|} < \frac{2}{n \|T\|}< 1$, which implies that for all $n\geq n_0$, $0< 1 - \frac{1}{n \|Tx_n\|}< 1$. \\
	Choose $\epsilon_n=\sqrt[]{1- (1- \frac{1}{n \|Tx_n\|})^2}$. Then clearly $\epsilon_n \rightarrow 0$ as $n \rightarrow \infty$.\\
	We claim that $Ax_n\in(Tx_n)^{+(\epsilon_n)}_{\alpha}$ for all $n\geq n_0$.\\
	Let $n \geq n_0$. Then for $0 \leq t < \frac{1}{n}$, \\
	$\|Tx_n + t \alpha Ax_n\| \geq \|Tx_n\|- t \|Ax_n\| \geq \|Tx_n\|-\frac{1}{n}$. \\
	For $\frac{1}{n} \leq t \leq n$, we claim that $\|Tx_n +t \alpha Ax_n\| \geq \|Tx_n\|- \frac{1}{n}$. \\
	Suppose on the contrary, we have, $\|Tx_n +t \alpha Ax_n\| < \|Tx_n\|- \frac{1}{n}$ for some $\frac{1}{n} \leq t \leq n$. Now, $Tx_n + \frac{\alpha}{n} Ax_n = (1- \frac{1}{n t})Tx_n + \frac{1}{n t}(Tx_n +t \alpha Ax_n)$. This implies that, $\|Tx_n\|- \frac{1}{n^3} < \|Tx_n + \frac{\alpha}{n} Ax_n\| \leq  (1- \frac{1}{n t})\|Tx_n\| + \frac{1}{n t}\|(Tx_n +t \alpha Ax_n)\| < (1- \frac{1}{n t})\|Tx_n\| + \frac{1}{n t}(\|Tx_n\| - \frac{1}{n})= \|Tx_n\|- \frac{1}{n^2 t}$. This implies that $t > n$, a contradiction. \\
	Thus, for $0 \leq t \leq n$, $\|Tx_n +t \alpha Ax_n\| \geq \|Tx_n\|- \frac{1}{n}$. Therefore, $0 \leq t \leq \frac{2 \|T\|}{c}$ gives that $\|Tx_n +t \alpha Ax_n\| \geq \|Tx_n\|- \frac{1}{n}$. \\
	Now, for $t > \frac{2 \|T\|}{c}$, $\|Tx_n +t \alpha Ax_n\| \geq t \|Ax_n\|- \|Tx_n\| \geq t c- \|Tx_n\| > 2\|T\|-\|Tx_n\| \geq \|Tx_n\|- \frac{1}{n}$. \\
	Therefore, for all $t \geq 0$, $\|Tx_n +t \alpha Ax_n\| \geq \|Tx_n\|- \frac{1}{n}= \sqrt[]{1-\epsilon_n^2} ~\|Tx_n\|$. This completes the proof of our claim. \\
	Similarly, considering $\|T-\frac{\alpha}{n} A\|> \|T\|- \frac{1}{n^3}$ for each $n \in \mathbb{N}$, we can find the desired sequences $\{y_n\} $ in $ S_{\mathbb{X}}$ and $ \{\delta _n\} $ in $ \mathbb{R}^+$ such that all the conditions of $(b)$ are satisfied. This completes the proof of the necessary part of the theorem.
\end{proof}

The corresponding characterization of Birkhoff-James orthogonality of bounded linear operators between real Banach spaces, now follow as a simple corollary to the previous theorem.
\begin{cor}(Theorem 2.4 of \cite{SPM})
	Let $\mathbb{X}$ and $\mathbb{Y}$ be two real Banach spaces. Let $T \in \mathbb{B} (\mathbb{X}, \mathbb{Y})$ be non-zero. Then for  any $A \in \mathbb{B} (\mathbb{X}, \mathbb{Y}),~ T \bot_B A$ if and only if either $(a)$ or $(b)$ holds: \\
	$(a)$ There exists a sequence $\{x_n\}$ in $S_\mathbb{X}$ such that $\| Tx_n\| \rightarrow \|T\|$ and $\|Ax_n\| \rightarrow 0$ as $n\rightarrow \infty$. \\
	$(b)$ There exist two sequences $\{x_n\},~\{y_n\}$ in $S_\mathbb{X}$ and two sequences of positive real numbers $\{\epsilon_n\}$, $\{\delta_n\}$ such that \\ 
	(i)  $ \epsilon_n \rightarrow 0$, $\delta_n\rightarrow 0$ as $n\rightarrow \infty$,\\
	(ii)  $\| Tx_n\| \rightarrow \|T\|$ and $\|Ty_n\|\rightarrow\|T\|$ as $n\rightarrow \infty$,\\
	(iii) $Ax_n\in(Tx_n)^{+(\epsilon_n)}$ and $Ay_n\in(Ty_n)^{-(\delta_n)}$ for all  $n\in \mathbb{N}$.
\end{cor}
\begin{proof}
	we note that in real Banach space, we must have, $\alpha \in U$ implies that $\alpha=1$. Therefore, the result follows easily from Theorem \ref{theorem:bounded}.
\end{proof}

In particular, in the statement of Theorem \ref{theorem:bounded}, if $\mathbb{X}$ is reflexive and $T,~ A$ from $\mathbb{X}$ to $\mathbb{Y}$ are compact, then we have the following theorem: 

\begin{theorem}\label{theorem:reflexive}
	Let $\mathbb{X}$ be a reflexive complex Banach space, $\mathbb{Y}$ be any complex Banach space. Let $ T, A \in \mathbb{K}(\mathbb{X}, \mathbb{Y}).$  Then $ T\perp_{B}A$ if and only if for each $\alpha \in U$ there exist $x = x(\alpha), ~y= y(\alpha) \in M_{T} $ such that $Ax \in (Tx)_\alpha^{+}$ and $Ay \in (Ty)_\alpha^{-}.$
\end{theorem}

\begin{proof}
	Suppose $T\bot_B A$. Then either $(a)$ or $(b)$ of Theorem \ref{theorem:bounded} holds. If $(a)$ holds then there exists a sequence $\{x_n\}$ such that $\|Tx_n\|\longrightarrow \|T\|$ and $\|Ax_n\|$ converges to $0$. Then since $\mathbb{X}$ is reflexive, $\{x_n\}$ has a weakly convergent subsequence, say $\{x_{n_k}\}$ converging weakly to $x$ (say). Since $T,A$ are compact operators, $\{Tx_{n_k}\},~ \{Ax_{n_k}\}$ converges to $Tx,~ Ax$ respectively. Hence $x\in M_T$ and $Ax=0$. Thus, $Ax \in (Tx)_\alpha^{+}$ and $Ax \in (Tx)_\alpha^{-}$ for each $\alpha \in U$. If $(b)$ holds, then by using similar arguments as in case $(a)$, we can  find $x=x(\alpha), ~y=(\alpha) \in M_T$ such that $Ax \in (Tx)_\alpha^{+}$ and $Ay \in (Ty)_\alpha^{-}.$
\end{proof}
In particular, if $ \mathbb{X}, \mathbb{Y} $ are finite-dimensional complex Banach spaces, then we have the following corollary:

\begin{cor}\label{corollary:ortho complex}
	Let $\mathbb{X}, \mathbb{Y}$ be finite-dimensional complex Banach spaces. Let $ T, A \in \mathbb{L}(\mathbb{X}, \mathbb{Y}).$  Then $ T\perp_{B}A$ if and only if for each $\alpha \in U$ there exist $x = x(\alpha), ~y= y(\alpha) \in M_{T} $ such that $Ax \in (Tx)_\alpha^{+}$ and $Ay \in (Ty)_\alpha^{-}.$  
\end{cor}

\begin{proof}
	Since every finite-dimensional complex Banach space is reflexive and every  linear operator on a finite-dimensional complex Banach space is compact, the proof of the corollary follows from Theorem \ref{theorem:reflexive}.
\end{proof}

We would further like to comment that the proofs of the corresponding characterization theorems in the real case are now obvious:
\begin{cor}\label{cor:reflexive}(Theorem 2.1 of \cite{SPM})
		Let $\mathbb{X}$ be a reflexive real Banach space, $\mathbb{Y}$ be any real Banach space. Let $ T, A \in \mathbb{K}(\mathbb{X}, \mathbb{Y}).$  Then $ T\perp_{B}A$ if and only if  there exist $x , y \in M_{T} $ such that $Ax \in (Tx)^{+}$ and $Ay \in (Ty)^{-}.$  
\end{cor}
\begin{proof}
	Let $T \bot_B A$. Since in real Banach space, $\alpha \in U$ implies that $\alpha=1$, by Theorem \ref{theorem:reflexive}, there exist $x,y \in M_T$ such that  $Ax \in (Tx)^{+}$ and $Ay \in (Ty)^{-}.$  
\end{proof}

\begin{cor}\label{corollary:ortho real}(Theorem 2.2 of \cite{S})
	Let $\mathbb{X}, \mathbb{Y}$ be finite-dimensional real Banach spaces. Let $ T, A \in \mathbb{L}(\mathbb{X}, \mathbb{Y}).$  Then $ T\perp_{B}A$ if and only if  there exist $x, y \in M_{T} $ such that $Ax \in (Tx)^{+}$ and $Ay \in (Ty)^{-}.$  
\end{cor}
\begin{proof}
		Since every finite-dimensional complex Banach space is reflexive and every  linear operator on a finite-dimensional complex Banach space is compact, the proof of the corollary follows from Corollary \ref{cor:reflexive}.
\end{proof}

\begin{remark}
In view of the method employed in proving Theorem $ 2.3, $ it should be mentioned that our proof essentially follows the same line of arguments, as outlined in the proof of Theorem $ 2.4 $ of \cite{SPM}. Indeed, the novelty lies in introducing the corresponding notations in the complex case, analogous to the real case. Moreover, as we will see in the next theorem, in spite of being a complete characterization of B-J orthogonality of compact linear operators on a reflexive complex Banach space, Theorem \ref{theorem:reflexive} does not capture the  full strength of the complex number system. Indeed, in our opinion, Theorem \ref{theorem:reflexive} should be regarded as a stepping stone towards our next theorem, that also distinguishes the complex case from the real case. First we need the following geometric lemma:
\end{remark}

\begin{lemma}\label{lemma:complex}
	Let $\mathbb{X}$ be a complex Banach space. Let $x,~y \in \mathbb{X}$ and $\alpha=e^{i \theta}$, where $\theta \in [0,\pi]$. If $y \in x^{+}_{\alpha}$ then either $y \in x^{+}_{\beta}$ for all $\beta$ with $\arg \beta \in [0, \theta]$ or $y \in x^{+}_{\beta}$ for all $\beta$ with  $\arg \beta \in [\theta, \pi]$.
\end{lemma}
\begin{proof}
Suppose  $y \notin x^{+}_{\alpha_1}$ for some $\alpha_1$ with  $\arg \alpha_1 \in [0, \theta]$. Then there exists $t_1 > 0$ such that $\|x+ t_1 \alpha_1 y\|< \|x\|$. We claim that $y \in x^{+}_{\beta}$ for all $\beta$ with $\arg \beta \in [\theta, \pi]$. If possible suppose that  $y \notin x^{+}_{\alpha_2}$ for some $\alpha_2$ with  $\arg \alpha_2 \in [ \theta, \pi]$. Then there exists $t_2 > 0$ such that $\|x+ t_2 \alpha_2 y\|< \|x\|$. Then it is easy to verify that there exist $0< s < 1$ and $t > 0$ such that $(1-s) t_1 \alpha_1 + s t_2 \alpha_2= t \alpha$. Therefore, $(1-s)[x+t_1 \alpha_1 y]+ s [x+ t_2 \alpha_2 y]= x+ t \alpha y$. This implies that $\|x+ t \alpha y\|\leq (1-s)\|x+ t_1 \alpha_1 y \|+ s \|x+ t_2 \alpha_2 y\|< (1-s)\|x\| + s \|x\|= \|x\|$, a contradiction. This proves our claim.
\end{proof}

Let us now prove the following characterization theorem, that improves the necessary part of Theorem \ref{theorem:reflexive}.

\begin{theorem}
	Let $\mathbb{X}$ be a reflexive complex Banach space and $\mathbb{Y}$ be any complex Banach space. Let $ T,A\in K(\mathbb{X},\mathbb{Y})$. Then $T\bot_{B}A$ if and only if there exist $x,~ y,~ z,~ w \in M_{T}$ and $\phi_{1}, \phi_{2} \in [0,\pi]$ such that\\
	 (i) $ Ax \in (Tx)^{+}_{\alpha} ~~\forall~~ \alpha $ with $\arg \alpha \in [0,\phi_{1}]$,\\
	(ii) $ Ay \in (Ty)^{+}_{\alpha} ~~\forall~~ \alpha$ with $\arg \alpha \in [\phi_{1}, \pi],$ \\
	(iii) $  Az \in (Tz)^{-}_{\alpha} ~~\forall~~ \alpha $ with $\arg \alpha \in [0,\phi_{2}]$, \\
	(iv) $Aw \in (Tw)^{-}_{\alpha} ~~\forall~~ \alpha $ with $\arg \alpha \in [\phi_{2}, \pi]$.
	
\end{theorem}	
\begin{proof}
	We first prove the easier sufficient part. Suppose there exist $x,~ y,~ z,~ w \in M_{T}$ and $\phi_{1}, \phi_{2} \in [0,\pi]$ such that all the conditions in $(i), (ii), (iii)$ and $(iv)$ are satisfied. Let $\lambda \in \mathbb{C}$. Then either there exist $t_1 \geq 0$ and $ \alpha_1 $ with $\arg \alpha_1 \in [0, \phi_1]$ such that $\lambda = t_1 \alpha_1$ or there exist $t_2 \geq 0$ and $\alpha_2 $ with $\arg \alpha_2  \in [ \phi_1, \pi]$ such that $\lambda = t_2 \alpha_2$ or there exist $t_3 \leq 0$ and $ \alpha_3 $ with $\arg \alpha_3 \in [0, \phi_2]$ such that $\lambda = t_3 \alpha_3$ or there exist $t_4 \leq 0$ and $ \alpha_4 $ with $\arg \alpha_4 \in [ \phi_2, \pi]$ such that $\lambda = t_4 \alpha_4$. Now $\lambda= t_1 \alpha_1$ implies that $\|T + \lambda A\| = \|T + t_1 \alpha_1 A\| \geq \|Tx + t_1 \alpha_1 Ax\| \geq \|Tx\|= \|T\|$. Similarly, in the other cases, it can be shown that $\|T+ \lambda A\|\geq \|T\|$.  Hence $T\bot_B A$.\\
    For the necessary part, suppose that $T\perp_{B}A$. Then from Theorem \ref{theorem:reflexive}, we have, for each $\alpha $ with $ \arg \alpha   \in [0,\pi]$, there exists $x_{\alpha}\in M_{T}$ such that $ Ax_{\alpha}\in (Tx_{\alpha})^{+}_{\alpha}$. 
	Now, consider 
	\[V_1=\{\theta \in [0, \pi]: \exists ~ x\in M_T : Ax \in (Tx)^{+}_{\alpha}~\forall ~ \alpha ~ with ~ \arg \alpha  \in [0,\theta]  \},\]
	\[V_2= \{\theta \in [0, \pi]: \exists ~ x\in M_T : Ax \in (Tx)^{+}_{\alpha}~\forall ~ \alpha ~ with ~ \arg \alpha  \in [\theta, \pi]  \}.\]
	Clearly, $ 0 \in V_1$ and $\pi \in V_2$ and therefore, $V_1,~ V_2$ are non-empty. Moreover, $V_1,~V_2 $ are bounded. Suppose $\xi=\sup V_1$ and $\eta= \inf V_2$. Now, we claim that $\xi \geq \eta$.  If possible suppose that $\xi < \eta$. Then consider $\zeta= \frac{\xi + \eta}{2} \in [0, \pi]$. Now, from Theorem \ref{theorem:reflexive}, we have, for $\alpha = e^{i \zeta}$ there exists $x_{\alpha} \in M_T$ such that $Ax_{\alpha}\in (Tx_{\alpha})^{+}_{\alpha}$. Using Lemma \ref{lemma:complex}, we have, either $Ax_{\alpha}\in (Tx_{\alpha})^{+}_{\beta}$ for all $\beta$ with $\arg \beta \in [0,\zeta]$ or $Ax_{\alpha}\in (Tx_{\alpha})^{+}_{\beta}$ for all $\beta$ with $\arg \beta \in [\zeta, \pi]$. But  $Ax_{\alpha}\in (Tx_{\alpha})^{+}_{\beta}$ for all $\beta$ with $ \arg \beta \in [0,\zeta]$ implies that $\zeta \in V_1$. This contradicts that $\xi= \sup V_1 $. Again, $Ax_{\alpha}\in (Tx_{\alpha})^{+}_{\beta}$ for all $\beta$ with $\arg \beta \in [\zeta, \pi]$ implies that $\zeta \in V_2$. This contradicts that $\eta= \inf V_2$. Hence $\xi \geq  \eta$.  Now, there exist sequences $\{\xi_n\} \subseteq V_1,~ \{\eta_{n} \}\subseteq V_2$ such that $\{\xi_n\}$ converges to $\xi$ and $\{\eta_{n}\}$ converges to $\eta$. Since $\xi_n \in V_1,~ \eta_{n} \in V_2$, there exist $x_{n},~ y_{{n}} \in M_T$ such that $Ax_{n} \in (Tx_{n})^{+}_{\alpha}~\forall ~ \alpha $ with $\arg \alpha \in [0,\xi_n] $ and $Ay_n  \in (Ty_n)^{+}_{\alpha}~\forall ~ \alpha $ with $\arg \alpha \in [\eta_{n}, \pi] $. Since $\mathbb{X}$ is reflexive, $\{x_{n}\},~ \{y_n \}$ have weakly convergent subsequences. Without loss of generality assume that $\{x_{n}\}$ weakly converges to $x$ and $\{y_n\}$  weakly converges to $y$. Since $T,A$ are compact operators, $Tx_n \longrightarrow Tx,~ Ty_n \longrightarrow Ty,~Ax_n \longrightarrow Ax,~ Ay_n \longrightarrow Ay.$ Clearly $x, y \in M_T$. Now, $\|Tx_{n}+ t \alpha Ax_{n}\|\geq \|T\|$ for all $t \geq 0$ and for all $\alpha$ with $ \arg \alpha \in [0, \xi_n] \Rightarrow \|Tx+ t \alpha Ax\|\geq \|T\|$ for all $t \geq 0$ and for all $\alpha$ with $\arg \alpha \in [0, \xi]$. Similarly, $\|Ty_n+ t \alpha Ay_n\|\geq \|T\|$ for all $t \geq 0$ and for all $\alpha$ with $\arg \alpha \in [ \eta_n, \pi] \Rightarrow \|Ty+ t \alpha Ay\|\geq \|T\|$ for all $t \geq 0$ and for all $\alpha$ with $\arg \alpha \in [ \eta, \pi]$. Since $\xi \geq \eta$, $\|Ty+ t \alpha Ay\|\geq \|T\|$ for all $t \geq 0$ and for all $\alpha$ with  $\arg \alpha \in [ \xi, \pi]$.  Let $\xi=\phi_{1}$. Then $Ax \in (Tx)_{\alpha}^+$ for all $\alpha$ with $\arg \alpha \in [0, \phi_1]$ and $Ay \in (Ty)_{\alpha}^+$ for all $\alpha $ with $\arg \alpha \in [\phi_1, \pi]$. \\
	Similarly, for each $\alpha$ with $\arg \alpha \in [0,\pi]$ there exists $z_{\alpha}\in M_{T}$ such that $ Az_{\alpha}\in (Tz_{\alpha})^{-}_{\alpha}$ gives that there exist $\phi_2 \in [0, \pi]$ and $z, w \in M_T$ such that 	$  Az \in (Tz)^{-}_{\alpha} ~~\forall~~ \alpha $ with $\arg \alpha \in [0,\phi_{2}]$ and
	$Aw \in (Tw)^{-}_{\alpha} ~~\forall~~ \alpha $ with $\arg \alpha \in [\phi_{2}, \pi]$.
\end{proof}

 Sain and Paul proved in \cite{SP} that if $T$ is a linear operator on  a finite-dimensional real Banach  space  $\mathbb{X}$, with $M_T = \pm D$ ($D$ being a closed connected subset of $S_{\mathbb{X}}$), then $ T \bot_{B} A $ if and only if there exists $ x \in D $ such that $Tx \bot_{B} Ax.$ In the following theorem we prove an analogous result for complex Banach spaces. Before proving the theorem, let us observe that if $\mathbb{X}$ is a complex Banach space, $ T \in \mathbb{L}(\mathbb{X})$ and $D$ is a closed connected subset of $S_{\mathbb{X}}$ such that $ D \subset M_T, $  then we must have, $ \bigcup_{\theta \in [0, 2\pi)} e^{i\theta}D \subset M_T $ and $ \bigcup_{\theta \in [0, 2\pi)} e^{i\theta}D $ is also a connected subset of $S_{\mathbb{X}}$. Note that, it is not true in general, if $\mathbb{X}$ is a real Banach space. This explains the change in the statement of Theorem \ref{theorem:connected}, compared to the corresponding real case. 

\begin{theorem}\label{theorem:connected}
Let $\mathbb{X}$ be a finite-dimensional complex Banach space. Let $ T \in \mathbb{L}(\mathbb{X})$ be such that $M_T$ is a closed connected subset of $S_{\mathbb{X}}.$ Then for $A \in \mathbb{L}(\mathbb{X})$, $ T \bot_B A$ if and only if for each $ \alpha \in U $, there exists $x = x(\alpha) \in M_T $ such that $ Tx \bot_{\alpha} Ax.$ 
\end{theorem}
\begin{proof} The sufficient part of the theorem follows trivially.
Since $ T \bot_B A $, applying Corollary \ref{corollary:ortho complex}, it follows that for each $\alpha \in U$, there exist $x = x(\alpha), ~y= y(\alpha) \in M_{T} $ such that $Ax \in (Tx)_\alpha^{+}$ and $Ay \in (Ty)_\alpha^{-}.$ 
 Then considering the line passing through $\alpha $ and $ - \alpha$ in the complex plane and  following the same arguments as in Theorem 2.1 of \cite{SP}, it can be concluded that there exists $u \in M_T$ such that $Au \in (Tu)_\alpha^+ $ and $ Au \in (Tu)_\alpha^-$, by using the connectedness of $M_T.$ This implies that $Tu \bot_{\alpha} Au.$ This establishes the theorem.
\end{proof} 

Once again, in contrast to the real case, we would like to sharpen the necessary part of Theorem \ref{theorem:connected} in the complex case. First we need the following lemma.

\begin{lemma}\label{lemma:connected}
	Let $\mathbb{X}$ be a complex Banach space. Let $x, y \in \mathbb{X} $ and $\alpha \in U$ with $\arg \alpha = \theta$ be such that  $x\bot_{\alpha} y$. Then either $y \in (x)^{+}_{\beta}$ for all $\beta$ with $\arg \beta \in [\theta- \pi, \theta]$ or $y \in (x)^{+}_{\beta}$ for all $\beta$ with  $\arg \beta \in [\theta, \theta + \pi]$. 
\end{lemma}

\begin{proof}
	Let  $x\bot_{\alpha} y$. Suppose that $y \notin (x)^{+}_{\beta_1}$ for some $\beta_1$ with $\arg \beta_1 \in [\theta- \pi, \theta]$. Then there exists $t_1 > 0$ such that $\|x+ t_1 \beta_1 y\|< \|x\|$. If possible suppose that $y \notin (x)^{+}_{\beta_2}$ for some $\beta_2$ with $\arg \beta_2 \in [\theta, \theta + \pi]$.  Then there exists $t_2 > 0$ such that $\|x+ t_2 \beta_2 y\|< \|x\|$. Then it is easy to verify that there exist $0<s<1$ and $t\in \mathbb{R}$ such that 
	\begin{eqnarray*}
	        t \alpha &=& (1-s)t_1 \beta_1+ s t_2 \beta_2 \\
	        \Rightarrow x+ t \alpha y &=&(1-s)(x+t_1 \beta_1 y)+ s (x+ t_2 \beta_2 y)  \\
	        \Rightarrow \|x+ t \alpha y\| &\leq &(1-s)\|(x+t_1 \beta_1 y)\|+ s \|(x+ t_2 \beta_2 y)\|  \\
	        \Rightarrow \|x+ t \alpha y\| &<&(1-s)\|x\|+ s \|x\| \\
	        &=&\|x\|,
	\end{eqnarray*}
	this leads to a contradiction and so $y \in (x)^{+}_{\beta}$ for all $\beta$ with $\arg \beta \in [\theta, \theta + \pi]$. This proves the lemma.
\end{proof}

Now, the promised theorem:
\begin{theorem}\label{theorem:conn}
   Let $\mathbb{X}$ be a finite-dimensional complex Banach space. Let $ T \in \mathbb{L}(\mathbb{X})$ be such that $M_T$ is a closed connected subset of $S_{\mathbb{X}}.$ Then for $A \in \mathbb{L}(\mathbb{X})$, $ T \bot_B A$ if and only if there exist some $ \theta \in [0,\pi] $ and $x ,y \in M_T $ such that $ Ax \in (Tx)^{+}_{\alpha}$ for all $\alpha$ with $\arg \alpha \in [\theta- \pi, \theta]$ and  $ Ay \in (Ty)^{+}_{\alpha}$ for all $\alpha$ with $\arg \alpha \in [\theta, \theta+ \pi]$.
\end{theorem}

\begin{proof}
	Let us first prove the sufficient part of the theorem. Suppose there exists some $ \theta \in [0,\pi] $ and $x ,y \in M_T$ such that $ Ax \in (Tx)^{+}_{\alpha}$ for all $\alpha$ with $\arg \alpha \in [\theta- \pi, \theta]$ and  $ Ay \in (Ty)^{+}_{\alpha}$ for all $\alpha$ with $\arg \alpha \in [\theta, \theta+ \pi]$. Let $\lambda \in \mathbb{C}$. Then either there exist $t_1 \geq  0$ and $ \alpha_1$ with $\arg \alpha_1 \in [\theta- \pi, \theta]$ such that $\lambda= t_1 \alpha_1$ or there exist $t_2 \geq  0$ and $\alpha_2$ with $\arg \alpha_2 \in [\theta, \theta + \pi]$ such that $\lambda= t_2 \alpha_2$. Now, $\lambda= t_1 \alpha_1 $ implies that $\|T+ \lambda A\|=\|T + t_1 \alpha_1 A\|\geq \|(T + t_1 \alpha_1 A)x\|\geq \|Tx\|=\|T\|$ and $\lambda= t_2 \alpha_2 $ implies that $\|T+ \lambda A\|=\|T + t_2 \alpha_2 A\|\geq \|(T + t_2 \alpha_2 A)y\|\geq \|Ty\|=\|T\|$. Therefore, $T\bot_B A$. This completes the proof of the sufficient part of the theorem.\\
	For the necessary part, suppose that $T\bot_B A$. Let us consider the following two sets
	\[V_1=\{\theta \in [0, \pi]: \exists~ x \in M_T ~:~ Ax ~ \in (Tx)^{+}_{\alpha}~ \forall ~ \alpha ~ with ~ \arg  \alpha  \in [\theta - \pi, \theta] \},\]
	 \[V_2=\{\theta \in [0, \pi]: \exists~ x \in M_T ~:~  Ax ~ \in (Tx)^{+}_{\alpha}~ \forall ~ \alpha ~ with ~ \arg \alpha \in [\theta, \theta+\pi]  \}.\]
	 We first show that $[0,\pi]= V_1 \cup V_2$. Let $\theta \in [0,\pi]$ and $\alpha= e^{i \theta}$. Since $T\bot_B A$, by Theorem \ref{theorem:connected}, we have, there exists $x=x(\alpha) \in M_T$ such that $Tx \bot_{\alpha} Ax$. Therefore, applying Lemma \ref{lemma:connected}, we have, either $Ax \in (Tx)^{+}_{\beta}$ for all $\beta$ with $\arg \beta \in [\theta- \pi, \theta]$, i.e, $\theta \in V_1$ or $Ax \in (Tx)^{+}_{\beta}$ for all $\beta$ with $\arg \beta \in [\theta, \theta + \pi]$, i.e, $\theta \in V_2$. Hence $[0,\pi]= V_1 \cup V_2$.\\
	  We claim that $V_1 \neq \emptyset$. Let $0 \notin V_1 $. Then $0 \in V_2$. Hence there exists $z \in M_T$ such that $Az ~ \in (Tz)^{+}_{\beta}~ \forall ~ \beta $ with $\arg \beta \in [0,\pi] $. This implies that $\pi \in V_1$. Hence $V_1 \neq \emptyset$. Similarly, it can be shown that $V_2 \neq \emptyset$.\\
	   We next show that $V_1$ is closed. Let $\{\theta_n\}$ be a sequence in $V_1$ such that $\{\theta_n\}$ converges to $\theta$. Let $\beta=e^{i \theta}$. Then there exists $x_n \in M_T$ such that $ Ax_n ~ \in (Tx_n)^{+}_{\alpha}~ \forall ~ \alpha$ with $\arg \alpha \in [\theta_n - \pi, \theta_n]$. Since $\mathbb{X}$ is finite-dimensional, $\{x_n\}$ has a convergent subsequence. Without loss of generality assume that $\{x_n\}$ converges to $x$ (say). Clearly, $x \in M_T$. Now, $Ax_n \in (Tx_n)^{+}_{\alpha}$ for all $\alpha$ with $\arg \alpha \in [\theta_n -\pi, \theta_n]$ gives that $\|Tx_n + t e^{i \theta_n} Ax_n\| \geq \|T\|$ for all $t \geq 0$. Letting $n \longrightarrow \infty$, we have $\|Tx+ t e^{i \theta}Ax\|=\|Tx + t\beta Ax\|\geq \|T\| \Rightarrow Ax \in (Tx)^{+}_{\beta}$. Similarly, $Ax \in (Tx)^{+}_{\gamma}$, where $\arg \gamma =\theta- \pi$. Now, let $\theta- \pi < \phi < \theta$. If possible suppose that there does not exist any $n_0 \in \mathbb{N}$ such that $\phi \in [\theta_n - \pi, \theta_n]$ for all $n \geq n_0$. Then without loss of generality we may assume that $\phi > \theta_n$ for all $n \in \mathbb{N}$. Letting $n \longrightarrow \infty$, we have $\phi \geq \theta$, a contradiction. Hence  there exists $n_0 \in \mathbb{N}$ such that $\phi \in [\theta_n - \pi, \theta_n]$ for all $n \geq n_0$. This implies that $\|Tx_n + t e^{i \phi}Ax_n\| \geq \|T\| $ for all $t \geq 0$ and for all $n \geq n_0$. Therefore, as $n \longrightarrow \infty$, we have $\|Tx+ te^{i \phi}Ax \| \geq \|T\|$ for all $t \geq 0$. This implies that $Ax \in (Tx)^{+}_{\delta}$, where $\delta= e^{i \phi}$.  Thus, $Ax\in (Tx)^{+}_{\alpha}$ for all $\alpha$ with $\arg \alpha \in [\theta -\pi, \theta]$. Hence $\theta \in V_1$. Thus, $V_1$ is closed. Similarly, it can be shown that $V_2$ is closed. \\
	   Now, since $[0,\pi]$ is connected, $V_1 \cap V_2 \neq \emptyset$. Let $\theta \in V_1 \cap V_2$. Then there exist $x, y \in M_T$ such that $ Ax ~ \in (Tx)^{+}_{\alpha}~ \forall ~ \alpha $ with $\arg \alpha \in [\theta - \pi, \theta]$ and $ Ay ~ \in (Ty)^{+}_{\alpha}~ \forall ~ \alpha$ with $\arg \alpha \in [\theta, \theta+ \pi]$. This establishes the theorem.  
\end{proof}

Next, in the context of complex Banach spaces we explore the structure of $M_T$ in connection with B-J orthogonality. We would like to invite the reader to have a look at Theorem 2.2 and Corollary 2.2.1 of \cite{Sa}, for an analogous result in the real case.
\begin{theorem}
Let $\mathbb{X} $ be a complex Banach space,  $0 \neq T\in \mathbb{L}(\mathbb{X})$  and  $x\in M_{T}.$ \\
(i)   If $y\in \mathbb{X}$ is such that $ Tx\perp_{B}Ty $ then $ x\perp_{B}y$.\\
(ii) $T(x_\alpha^{+}\setminus x_\alpha^{\perp})\subset (Tx)_\alpha^{+}\setminus (Tx)_\alpha^{\perp}$, for $\alpha \in U.$\\
(iii)  $T(x_\alpha^{-}\setminus x_\alpha^{\perp})\subset (Tx)_\alpha^{-}\setminus (Tx)_\alpha^{\perp}$, for $\alpha \in U.$\\
(iv) $\ker T \subset \bigcap_{x\in M_{T}}x^{\perp}$.
\end{theorem}	
\begin{proof}
$(i)$ 	Suppose $ Tx\perp_{B}Ty $. Then $\|T\|\|x\|= \|Tx\| \leq \|Tx+\lambda Ty\| \leq \|T\| \|x+\lambda y\| ~\forall \lambda \in \mathbb{C}$. This implies that $ \| x + \lambda y \| \geq \|x\| ~\forall \lambda \in \mathbb{C}$. Therefore, $ x \bot_B y.$ \\ 

$(ii)$  Let $y \in x_\alpha^{+}\setminus x_\alpha^{\perp}.$ Then there exist some $t<0$ such that $ \|x+ t \alpha y\|< \|x\|.$ Now, $\|Tx + t \alpha Ty\| \leq \|T\| \|x+ t \alpha y\|< \|T\|\|x\|=\|Tx\|$. This implies that $  Ty\notin (Tx)_{\alpha}^{-}$. It now follows from Proposition 2.1 that $Ty\in (Tx)_\alpha^{+}\setminus (Tx)_\alpha^{\perp}$. \\

$(iii)$ Follows similarly as $(ii)$. \\
	
$(iv)$ If   $M_{T}=\phi$, then the theorem follows trivially. Let us assume that  $M_{T}\neq\phi$. Let $y\in \ker T.$ Then for any $x\in M_{T}$, we have $Ty\in(Tx)^{\perp}$, since $Ty=0$. From (i) it follows that $y\in x^{\perp}$. This implies that 
$\ker T \subset \bigcap_{x\in M_{T}}x^{\perp}$.
\end{proof}

\begin{remark}
In addition, if $x,~Tx$ are smooth points in $\mathbb{X}$ then for any $y \in \mathbb{X}$, we have, $x\bot_B y \Rightarrow Tx \bot_B Ty$. This can be proved following the same line of arguments, as in Lemma 2.1 of \cite{SGP}.

\end{remark}

We next study left symmetric linear operator(s) defined on a finite-dimensional complex Banach space $\mathbb{X}.$ An element $ T \in \mathbb{L}(\mathbb{X}) $ is said to be left symmetric if $ T \bot_B A $ implies $ A \bot_B T $ for any $ A \in \mathbb{L}(\mathbb{X}).$ In \cite{S}, Sain proved a nice connection between the left symmetry of bounded linear operators and the left symmetric points in the corresponding norm attainment set, for real Banach spaces. Following the same line of arguments, these results can be proved for complex Banach spaces. We simply state the following two theorems for complex Banach spaces.
\begin{theorem}\label{theorem:left1}
Let $\mathbb{X} $ be a finite-dimensional strictly convex complex Banach space. If $ T \in \mathbb{L}(\mathbb{X})$ is left symmetric then for each $x \in M_T$, $Tx$ is a left symmetric point.
\end{theorem}
\begin{theorem}\label{theorem:left2}
	Let $\mathbb{X}$ be a finite-dimensional strictly convex and smooth complex Banach space. Let $ T \in \mathbb{L}(\mathbb{X})$
be such that there exists $x, y \in S_\mathbb{X}$ satisfying\\
 $(i)$ $x\in M_{T}$, $(ii)$ $y\perp_{B} x$, $(iii)$ $Ty\neq 0$. Then $T$ can not be left symmetric.
\end{theorem}
In \cite{GSP}, it was proved that if $T$ is a compact linear operator on a real Hilbert space $\mathbb{H}$ then $T$ is left symmetric if and only if $T$ is the zero operator. Sain proved in \cite{S} that if $T \in \ell_p^2(\mathbb{R}) , 1 < p < \infty,$ then $T$ is left symmetric if and only if $T$ is the zero operator. The situation is a little bit complicated when we consider $ \ell_p^2(\mathbb{C})$ instead of $\ell_p^2(\mathbb{R})$. We first find the symmetric points of $ \ell_p^2(\mathbb{C})$ in the following propositions. We would like to mention that the proofs of the propositions follow from elementary calculations in each case.  

\begin{prop}
Let $\mathbb{X}=\ell_{p}^{2}(\mathbb{C}),1 < p <  \infty$. If $(z_{1},z_{2}) \in S_{\mathbb{X}}$ then \\
(i) $(z_{1},z_{2})\perp_{B} (1,-\frac{|z_{1}|^{p-2}\bar{z_{1}}}{|z_{2}|^{p-2}\bar{z_{2}}})$, where  $ z_{2} \neq 0.$ \\
(ii)  $(z_{1},z_{2})\perp_{B} (-\frac{|z_{2}|^{p-2}\bar{z_{2}}}{|z_{1}|^{p-2}\bar{z_{1}}}, 1)$, where  $ z_{1} \neq 0.$  
\end{prop}

\begin{proof}
	$(i)$ Let $z_2 \neq 0$. Then it is easy to see from elementary calculations that the Gateaux derivative $G((z_{1},z_{2}),(1,-\dfrac{|z_{1}|^{p-2}\bar{z_{1}}}{|z_{2}|^{p-2}\bar{z_{2}}}))=0$. Therefore, $(z_{1},z_{2})\perp_{B} (1,-\frac{|z_{1}|^{p-2}\bar{z_{1}}}{|z_{2}|^{p-2}\bar{z_{2}}})$, where  $ z_{2} \neq 0.$ \\
	$(ii)$ Follows similarly as $(i)$.
\end{proof}

\begin{prop}
Let $\mathbb{X}=\ell_{p}^{2}(\mathbb{C}), 1 < p <  \infty$. If $ x, y\in S_{\mathbb{X}}$ are such that $x\perp_{B}y $ and $y\perp_{B}x, $ then either of the following is true:
  	
  	(i) $x=(e^{i\theta_{1}},0)$ and $y= (0,e^{i\theta_{2}})$, where $\theta_i \in [0, 2\pi), i=1,2$.
  	
  	(ii) $x= (0,e^{i\theta_{1}})$ and $y=(e^{i\theta_{2}},0)$, where $\theta_i \in [0, 2\pi), i=1,2$.
  	
  	(iii) $x=(\dfrac{1}{2^{1/p}} e^{i\theta_{1}},\dfrac{1}{2^{1/p}} e^{i\theta_{2}})$ and $y= (\dfrac{1}{2^{1/p}} e^{i\theta_{1}},-\dfrac{1}{2^{1/p}} e^{i\theta_{2}})$, where $\theta_i \in [0, 2\pi), i=1,2$.
  	
  	(iv) $x=(\dfrac{1}{2^{1/p}} e^{i\theta_{1}},-\dfrac{1}{2^{1/p}} e^{i\theta_{2}})$ and $y= (\dfrac{1}{2^{1/p}} e^{i\theta_{1}},\dfrac{1}{2^{1/p}} e^{i\theta_{2}})$, where $\theta_i \in [0, 2\pi), i=1,2$.
\end{prop}

\begin{prop}\label{prop:left}
	Let $\mathbb{X}$ be  the two-dimensional complex Banach space $l_{p}^{2}(\mathbb{C}), 1< p < \infty$. Then $x\in S_{\mathbb{X}}$ is a left symmetric point in $\mathbb{X}$ if and only if 
	
	$x\in \{(e^{i\theta_{1}},0),(0,e^{i\theta_{2}}),(\dfrac{1}{2^{1/p}} e^{i\theta_{3}},\dfrac{1}{2^{1/p}} e^{i\theta_{4}}),(\dfrac{1}{2^{1/p}} e^{i\theta_{3}},-\dfrac{1}{2^{1/p}} e^{i\theta_{4}}) \}$, where $\theta_i \in [0, 2\pi), i=1,2,3,4$.
\end{prop}

\begin{theorem}
	Let $\mathbb{X}=\ell_{p}^{2}(\mathbb{C}),1 < p <  \infty$. Then $T\in \mathbb{L}({\mathbb{X}})$ is left symmetric if and only if $T$ is the zero operator.
\end{theorem}
\begin{proof}
	If possible suppose that $T\in \mathbb{L}({\mathbb{X}})$ is a non-zero left symmetric point. Since B-J orthogonality is homogeneous and $T$ is non-zero, we may assume without loss of generality, that $\|T\|=1$. Suppose $T$ attains norm at $x\in S_{\mathbb{X}}$. From Theorem 2.3 of James \cite{J}, it follows that there exists $y\in S_{\mathbb{X}}$  such that $y\perp_{B}x$. Since $\mathbb{X}$ is strictly convex and smooth, applying Theorem \ref{theorem:left2}, we see that $Ty=0$. Theorem \ref{theorem:left1}, ensures that $Tx$ must be a left symmetric point in $\mathbb{X}$. Thus, applying Proposition \ref{prop:left}, we have that
	
	$Tx\in \{(e^{i\theta_{1}},0),(0,e^{i\theta_{2}}),(\dfrac{1}{2^{1/p}} e^{i\theta_{3}},\dfrac{1}{2^{1/p}} e^{i\theta_{4}}),(\dfrac{1}{2^{1/p}} e^{i\theta_{3}},-\dfrac{1}{2^{1/p}} e^{i\theta_{4}}) \}$, where $\theta_i \in [0, 2\pi),~ i=1,2,3,4$.
	
	We claim that $x\perp_{B}y$.
	
	From Theorem 2.3 of James \cite{J}, it follows that there exists $\alpha\in\mathbb{C} $ such that $\alpha y+x\perp_{B}y$. Since $y\perp_{B}x$ and $x, y\neq0$, $\{x,y\}$ is linearly independent and hence $\alpha y+x\neq0$. Consider $ z=\dfrac{\alpha y+x}{\|\alpha y+x\|}$. Now, if $Tz=0$ then $T$ is zero operator. Let $Tz\neq0$. Clearly, $\{y,z\}$ is a basis of $\mathbb{X}$.
	
	Let $\|c_{1}z+c_{2}y\|=1$ for some $c_{1}, c_{2}\in\mathbb{C}$. Then $1=\|c_{1}z+c_{2}y\|=|c_{1}|\|z+\dfrac{c_{2}}{c_{1}}y\|\geq|c_{1}|$. Since $\mathbb{X}$ is strictly convex, $|c_{1}|<1$, if $c_{2}\neq0$. We also have, $\|T(c_{1}z+c_{2}y)\|=\|c_{1}Tz\|\leq\|Tz\|$ and $\|T(c_{1}z+c_{2}y)\|=\|Tz\|$  if and only if $|c_{1}|=1$ and $c_{2}=0$. This shows that $M_{T}=\{e^{i \theta} z: \theta \in [0,2 \pi)\}$.Thus, we must have $x= e^{i \theta}z$ for some $\theta \in [0,2 \pi)$. Hence $x\perp_{B}y$.
	
	Thus, $x, y\in\mathbb{X}$ are such that $x\perp_{B}y$ and $y\perp_{B}x$. Therefore, by Proposition \ref{prop:left}, we see that  we have the following informations about $T$:
	
	$(i)$ $T$ attains norm at $x,~ x\perp_{B}y,~ y\perp_{B}x,~ Ty =0$,~ $Tx$ is left symmetric.
	
	$(ii)$ $ x,~ y,~ Tx\in\{(e^{i\theta_{1}},0),(0,e^{i\theta_{2}}),(\dfrac{1}{2^{1/p}} e^{i\theta_{3}},\dfrac{1}{2^{1/p}} e^{i\theta_{4}}),(\dfrac{1}{2^{1/p}} e^{i\theta_{3}},-\dfrac{1}{2^{1/p}} e^{i\theta_{4}})\}$, where $\theta_i \in [0, 2\pi), i=1,2,3,4$.
	
	In order to prove that $T\in\mathbb{L}({\mathbb{X}})$ is left symmetric if and only if $T$ is the zero operator, we only need to consider 16 different types of operators that satisfy $(i)$ and $(ii)$ and show that none of them is left symmetric.
	
	Let us first consider one such typical linear operator and prove that it is not left symmetric.
	
	Let $T\in\mathbb{L}({\mathbb{X}})$ be defined by $T(e^{i\theta},0)=(e^{i\theta},0),~ T(0,e^{i\theta})=(0,0)$. Define $A\in\mathbb{L}(\mathbb{X})$ by $A(e^{i\theta},0)=(0,e^{i\theta}),~ A(0,e^{i\theta})=(e^{i\theta},e^{i\theta})$. Then $M_{T}= (e^{i\theta_{1}},0)$. Since $T(e^{i\theta_{1}},0)=(e^{i\theta_{1}},0)\perp_{B} (0,e^{i\theta_{1}})=A(e^{i\theta_{1}},0)$, it follows that $T\perp_{B}A$. We claim that $A\not\perp_{B}T$.
	
	Now, $\|A(\dfrac{1}{2^{1/p}} e^{i\theta},\dfrac{1}{2^{1/p}} e^{i\theta})\|^p=\|(\dfrac{1}{2^{1/p}}e^{i\theta},\dfrac{1}{2^{1/p}}e^{i\theta}+\dfrac{1}{2^{1/p}}e^{i\theta})\|^p=\dfrac{1}{2}+\dfrac{1}{2}|1+1|^p=\dfrac{1}{2} +2^{p-1}>2=\|A(0,e^{i\theta_{2}})\|^p$. This proves that $ (e^{i\theta_{1}},0), (0,e^{i\theta_{2}})\notin M_{A}$.
	
	Let $(\beta_{1} ,\beta_{2})\in M_{A}$, where $\beta_{1}=a+ib, \beta_{2}=c+id$. Since $ \|A(\beta_{1} ,\beta_{2})\|^p=|\beta_{2}|^p + |\beta_{1} + \beta_{2}|^p =|c+id|^p + |(a+c)+i(b+d)|^p $, one of the following holds: $(i)$ $a,b,c,d >0$,  $(ii) a,b,c,d<0$, $(iii) a,c>0, b,d<0$, $(iv) a,c<0, b,d>0$. Suppose $\lambda=\alpha t$, where $\alpha=\alpha_{1} + i\alpha_{2}$, with $\alpha \in U$ and $t \in \mathbb{R}$. Now, $ \|A(\beta_{1} ,\beta_{2})+\lambda T(\beta_{1} ,\beta_{2})\|^p=\|(\beta_{2}, \beta_{1}+\beta_{2}), \lambda(\beta_{1},0)\|^p=|\beta_{2}+\lambda\beta_{1}|^P+|\beta_{1}+\beta_{2}|^p =[(c+\alpha_{1}ta-\alpha_{2}tb)^2 +(d+\alpha_{1}tb+\alpha_{2}ta)^2]^\frac{p}{2}+ [(a+c)^2+(b+d)^2]^\frac{p}{2}$. If we take $t < - \frac {2[\alpha_{1}(ac+bd)+\alpha_{2}(ad-bc)]}{|\beta_{1}|^2}$, we have, $ (c+\alpha_{1}ta-\alpha_{2}tb)^2 +(d+\alpha_{1}tb+\alpha_{2}ta)^2 < c^2 + d^2$. Therefore, for $t < - \frac {2[\alpha_{1}(ac+bd)+\alpha_{2}(ad-bc)]}{|\beta_{1}|^2}$, we have, $ \|A(\beta_{1} ,\beta_{2})+\lambda T(\beta_{1} ,\beta_{2})\|^p=\|(\beta_{2}, \beta_{1}+\beta_{2}), \lambda(\beta_{1},0)\|^p=|\beta_{2}+\lambda\beta_{1}|^P+|\beta_{1}+\beta_{2}|^p<|\beta_{2}|^p+|\beta_{1}+\beta_{2}|^p= \|A(\beta_{1} ,\beta_{2})\|^p$.
	
	 So we can find $\lambda= \alpha t, ~\alpha \in U$, with $t < 0$ such that $\|A(\beta_{1} ,\beta_{2})+\lambda T(\beta_{1} ,\beta_{2})\|^p=\|(\beta_{2}, \beta_{1}+\beta_{2}), \lambda(\beta_{1},0)\|^p=|\beta_{2}+\lambda\beta_{1}|^P+|\beta_{1}+\beta_{2}|^p<|\beta_{2}|^p+|\beta_{1}+\beta_{2}|^p$. This proves that for any $w\in M_{A},~ Tw\notin (Aw)_\alpha^{-}$ for each $\alpha \in U$. Applying Corollary \ref{corollary:ortho complex}, it now follows that $A\not\perp_B T$. However, this proves that $T$ is not left symmetric in $\mathbb{L}(\mathbb{X})$, contradicting our initial assumption.
	
	Next, we describe a general method to prove that none among these 16 types of linear operators are left symmetric.
	
	Let $T$ attains norm at $x, x\perp_{B}y, y\perp_{B}x, Ty =0$ and
	
	 $ x,~y,~ Tx\in\{(e^{i\theta_{1}},0),(0,e^{i\theta_{2}}),(\dfrac{1}{2^{1/p}} e^{i\theta_{3}},\dfrac{1}{2^{1/p}} e^{i\theta_{4}}),(\dfrac{1}{2^{1/p}} e^{i\theta_{3}},-\dfrac{1}{2^{1/p}} e^{i\theta_{4}})\}$, where $\theta_i \in [0, 2\pi), i=1,2,3,4$.
	 
	 Define a linear operator $A\in \mathbb{L}(\mathbb{X})$ by $Ax=y,~ Ay=(e^{i\theta},0)$ or $(e^{i\theta},e^{i\theta})$ such that the following two conditions are satisfied:
	 
	 $(i) A$ does not attain its norm at $ x,~  y$.
	 $(ii)Tw\notin (Aw)_\alpha^{-}$ for each $\alpha \in U$ and for any $w\in M_{A}$.
	 
	 Then as before it is easy to see that $T\perp_{B}A$ but $A\not\perp_{B}T$. Thus, $T$ is not left symmetric. This completes the proof of the theorem.
	
\end{proof}

\begin{remark}
	The general method described in the previous theorem to prove that $T \in \mathbb{L}(\ell_p^2({\mathbb{C}}))(1<p < \infty)$ is left symmetric if and only if $T$ is the zero operator has been verified separately in each possible case mentioned in the proof of the theorem. In this paper, only one particular case has been dealt with explicitly. The details have been omitted in other cases since the method remains same in each case.  
\end{remark}
\begin{acknowledgement}
Dr. Debmalya Sain would like to lovingly acknowledge the monumental role played by Shri Sakti Prasad Mishra, a National Award winning teacher from his school Ramakrishna Mission Vidyapith, Purulia, in rightly shaping the philosophies of so many young learners including himself. 

\end{acknowledgement}

\end{document}